\documentclass[11pt,height = 615pt,width = 360pt]{amsart}
\numberwithin{equation}{section}

\usepackage{bbm}
\usepackage{comment}
\usepackage{verbatim}
\usepackage{mathrsfs}
\usepackage{mathtools}
\usepackage{latexsym}
\usepackage{epsfig}
\usepackage{amsmath}
\usepackage[usenames,dvipsnames]{xcolor}
\usepackage{bbm}
\usepackage{graphics}
\usepackage[utf8]{inputenc}
\usepackage{amssymb}
\usepackage{amsthm}
\usepackage[alphabetic]{amsrefs}
\usepackage{amsopn}
\usepackage{amscd}
\usepackage[all,knot]{xy}
\xyoption{all}
\usepackage{rotating}
\usepackage{tikz}
\usetikzlibrary{calc}
\usepgflibrary{shapes.geometric}
\usepgflibrary{shapes.misc}
\usetikzlibrary{positioning}
\usetikzlibrary{decorations}
\usetikzlibrary{decorations.pathreplacing}
\usepackage{hyperref}
\usepackage{tikz}
\usepackage{tikz-cd}
\mathtoolsset{showonlyrefs}

\theoremstyle{plain}
\newtheorem{thm}{Theorem}[section]
\newtheorem{lem}[thm]{Lemma}
\newtheorem{prop}[thm]{Proposition}
\newtheorem{cor}[thm]{Corollary}

\newtheorem*{thm*}{Theorem}
\newtheorem*{lem*}{Lemma}
\newtheorem*{prop*}{Proposition}
\newtheorem*{cor*}{Corollary}

\theoremstyle{definition}
\newtheorem{defn}[thm]{Definition}
\newtheorem*{defn*}{Definition}

\newtheorem{ex}[thm]{Example}
{}
\newtheorem{rem}[thm]{Remark}
\newtheorem*{rem*}{Remark}

{}
{}
{}
\newtheorem*{ack}{Acknowledgements}{}

\theoremstyle{remark}
{}
{}
{}

\def\to{\longrightarrow} 

\def\NN{\mathbb{N}}

\def\ZZ{\mathbb{Z}}
\DeclareUnicodeCharacter{221E}{$\infty$}

\def\sfD{\mathsf{D}}

\def\sfH{\mathsf{H}}

\def\mcA{\mathcal{A}}
\def\mcB{\mathcal{B}}

\def\mcF{\mathcal{F}}

\def\op{\mathrm{op}}
\def\b{\mathrm{b}}

\def\base{k}

\def\A{\Lambda}
\def\DA{\mathcal{E}}
\def\AC{\mathcal{S}}

\DeclareMathOperator{\add}{add}
\DeclareMathOperator{\Add}{Add}

\DeclareMathOperator{\Hom}{Hom}

\DeclareMathOperator{\End}{End}
\DeclareMathOperator{\id}{id}

\DeclareMathOperator{\modu}{\mathsf{mod}}
\DeclareMathOperator{\Modu}{\mathsf{Mod}}

\DeclareMathOperator{\proj}{\mathsf{proj}}
\DeclareMathOperator{\Perf}{\mathsf{Perf}}

\DeclareMathOperator{\RHom}{\mathbf{R}Hom}

\DeclareMathOperator{\coh}{coh}

\DeclareMathOperator{\loc}{\mathrm{loc}}
\DeclareMathOperator{\thick}{thick}

\DeclareMathOperator{\rad}{rad}

\let\cal\mathcal
\def\Ascr{{\cal A}}
\def\Bscr{{\cal B}}

\let\blb\mathbb

\def \ZZ{{\blb Z}}

\def \NN{{\blb N}}

\def\proj{\operatorname{proj}}
\def\id{\text{id}}

\def\coh{\mathop{\text{\upshape{coh}}}}

\def\rad{\operatorname {rad}}

\def\Hom{\operatorname {Hom}}

\def\End{\operatorname {End}}
\def\RHom{\operatorname {RHom}}

\def\ker{\operatorname {ker}}

\def\End{\operatorname {End}}

\def\id{{\operatorname {id}}}

\def\add{\operatorname {add}}

\def\r{\rightarrow}

\newdimen\uboxsep \uboxsep=1ex
\def\uboxn#1{\vtop to 0pt{\hrule height 0pt depth 0pt\vskip\uboxsep
\hbox to 0pt{\hss #1\hss}\vss}}

\def\uboxs#1{\vbox to 0pt{\vss\hbox to 0pt{\hss #1\hss}
\vskip\uboxsep\hrule height 0pt depth 0pt}}

\let\oldmarginpar\marginpar
\long\def\marginpar#1{\oldmarginpar{\raggedright \tiny \baselineskip 5pt #1}}

\definecolor{ruta2}{rgb}{0.409, 0.459, 0.208}
\usepackage{xcolor}
\hypersetup{
    colorlinks,
    linkcolor={red!50!black},
    citecolor={blue!50!black},
    urlcolor={blue!80!black}
}

\def\Perf{\operatorname{Perf}}
\DeclareMathOperator\PERF{\underline{\mathrm{Perf}}}

\def\aa{\mathfrak{a}}
\def\bb{\mathfrak{b}}

\definecolor{internationalkleinblue}{rgb}{0.0, 0.18, 0.65}

\title[Proper connective differential graded algebras \ldots]{Proper connective differential graded algebras and their geometric realizations}

\author{Theo Raedschelders}
\thanks{The first author is supported by a postdoctoral fellowship from the Research Foundation - Flanders (FWO)}
\address{Theo Raedschelders, Departement Wiskunde, Vrije Universiteit Brussel, 
Pleinlaan 2,
B-1050 Elsene
}
\email{Theo.Raedschelders@vub.be}
\urladdr{https://www.theoraedschelders.com/}

\author{Greg Stevenson}
\address{Greg Stevenson, Department of Mathematics,
Aarhus University,
Ny Munkegade 118, bldg. 1530
DK-8000 Aarhus C
Denmark
}
\email{greg@math.au.dk}
\urladdr{https://sites.google.com/view/gregstevenson}


\begin{document}

\begin{abstract}
\noindent 
We show that every proper connective dg algebra $A$ admits a geometric realization (as defined by Orlov) by a smooth projective scheme with a full exceptional collection. If $A$ is moreover smooth, we compute the noncommutative Chow motive of $A$. We go on to analyse the relationship between smoothness and regularity in more detail as well as commenting on smoothness of the degree zero cohomology for smooth proper connective dg algebras.
\end{abstract}

\maketitle

\tableofcontents



\section{Introduction}

By now dg algebras (differential graded algebras, if you're not into the whole brevity thing) form a fundamental part of modern mathematics; they control, and provide convenient models for, the majority of homological information arising from a diverse array of areas ranging from representation theory to symplectic and algebraic geometry. Moreover, one can and should think of dg algebras as giving rise to geometry in their own right: they are derived noncommutative spaces. 

In order to make more precise statements let us fix an algebraically closed field $k$ and let $A$ be a proper dg algebra over $k$ (in other words $A$ has finite dimensional total cohomology over $k$). The easiest such dg algebras to handle are the \emph{connective} ones, i.e.\ those with $\sfH^i(A) = 0$ for $i>0$ or equivalently, up to quasi-isomorphism, $A^i=0$ for $i>0$. This is still a quite general class of dg algebras and is rather `geometric' in the sense that it is equivalent to considering suitable simplicial algebras via the Dold-Kan correspondence.

The simplest class of proper connective dg algebras is given by the finite dimensional algebras, considered as dg algebras concentrated in degree zero. The corresponding derived noncommutative spaces are known to belong to the world of commutative geometry. Indeed, by a remarkable insight of Orlov \cite{MR3545926}, building on work of Auslander, for any finite dimensional algebra $\Lambda$ there is an embedding of $\Perf\Lambda$ into the bounded derived category of coherent sheaves on a smooth projective scheme, which can even be assumed to have a full exceptional collection. Orlov calls such an embedding a \emph{geometric realization} of $\Lambda$.

Proper connective dg algebras behave like finite dimensional algebras in many respects. For instance, their Grothendieck group is free of finite rank and the unbounded derived category has a standard t-structure with heart the category of modules for a finite dimensional algebra. It is thus natural to ask if one can also find the corresponding derived noncommutative schemes lurking in the commutative setting. Our main result is Theorem~\ref{thm:main} which asserts that this is indeed the case, for every proper connective dg algebra $A$ one can construct a geometric realization
\begin{displaymath}
\Perf A \hookrightarrow \sfD^\b(\coh X)
\end{displaymath}
where $X$ is a smooth projective scheme with a full exceptional collection, answering a question of Orlov \cite[Question 4.4]{MR3545926} in the connective case.

This result has a number of strong consequences. For instance, it implies that $A$ is quasi-isomorphic to a dg algebra which is finite dimensional on the nose (Corollary~\ref{cor:fd}). It also provides a conceptual justification for a number of observations concerning additive invariants of smooth proper connective dg algebras: by Corollary~\ref{cor:motive}, for a smooth proper connective dg algebra A, the noncommutative Chow motive of $\Perf A$ is a sum of points and thus all additive invariants are completely determined by the number of simples for $\sfH^0(A)$. A direct proof of this fact, relaxing the assumption that $k$ be algebraically closed, is also given in Theorem~\ref{th:motives}.

The remainder of the paper is essentially devoted to results and cautionary tales on smoothness and regularity for the perfect complexes and bounded derived category of a proper connective dg algebra. In particular, we show that, under reasonable circumstances, smoothness implies regularity for dg algebras over a base of finite global dimension (Proposition~\ref{prop:regular} which extends work of Lunts \cite{lunts2010categorical}), give a criterion for a proper connective dg algebra to be smooth in terms of its simple modules, (Theorem~\ref{thm:criterion}), and discuss the failure of $\sfH^0(A)$ to be smooth.



\section{Preliminaries and conventions}\label{sec:prelims}
Fix a base field $k$ and let $A=(\bigoplus_{i\in \mathbb{Z}} A^i,d)$ denote a dg algebra over $\base$. Note that we work with cohomological grading, so $d$ is of degree $1$. A \textit{dg ideal} $I$ of $A$ is a graded ideal such that $d(I) \subset I$. 

We say $A$ is \textit{proper} if $A$ is perfect as a dg module over $\base$, and that $A$ is \textit{smooth} if $A$ is perfect as a dg module over $A^e=A^{\op} \otimes_{\base} A$. We say $A$ is \textit{finite dimensional} if $A$ is a bounded complex of finite dimensional vector spaces on the nose (often we will abuse this terminology by calling $A$ finite dimensional if it is quasi-isomorphic to such a dga). We note that this is a strictly stronger condition than properness.

The dg algebra $A$ is \textit{connective} if $\mathsf{H}^i(A)=0$ for $i>0$. One can always replace a connective dg algebra by a quasi-isomorphic one which is zero in positive degrees; this is accomplished by noting that the good truncation in non-positive degrees is a subalgebra. 

We will denote by $\sfD(A)$ the derived category of right dg $A$-modules and by $\Perf A$ the (triangulated) category of perfect right dg $A$-modules. The category $\sf{D}(A)$ is compactly generated by $A$. Since $\base$ is a field and $A$ is proper, $\Perf A$ is Krull-Schmidt, i.e.\ every object decomposes into a finite direct sum of objects with local endomorphism rings. The Grothendieck group $K_0(A)$ of $A$ is, by definition, the Grothendieck group of the triangulated category $\Perf A$. 

We will also make use of $A_\infty$-categories and $A_{\infty}$-algebras, since we require a construction from \cite{Maaseik}. We assume that all $A_{\infty}$-notions are strictly unital and refer to loc.\ cit. for unexplained notation and further references. As for dg algebras we will call an $A_\infty$-algebra connective if it is concentrated in non-positive degrees, and we will work with cohomological grading throughout.

Given a finite dimensional algebra $\Lambda$ over $k$ we write $\rad(\Lambda)$ for its Jacobson radical.

\textbf{Standing assumptions:}
in the remainder of this paper we will assume that $A$ is a connective dg algebra over $\base$, such that $A^i=0$ for $i>0$, unless explicitly stated otherwise. All functors are derived and we omit the cumbersome decorations that indicate this. 




\section{Geometric realizations (following Orlov)}\label{sec:georeals}

This section is concerned with the problem of embedding a derived noncommutative scheme into the derived category of a (discrete, commutative) scheme. The set-up we wish to consider is made precise in the following definition, adapted from work of Orlov \cite{Orlov18}*{Definition~2.17} (cf.\ Orlov's earlier work \cites{MR3545926, Orlov15} where this notion is developed for derived noncommutative schemes with full exceptional collections). 

\begin{defn}\label{defn:gr}
A dg algebra $A$ admits a \emph{geometric realization} if there exists a fully faithful functor
\begin{equation}
r\colon \Perf A \to D^b(\coh(X)),
\end{equation}
for some smooth projective scheme $X$, which can be enhanced i.e.\ one can choose appropriate dg models and a dg functor between them such that taking homotopy yields the specified functor $r$.
\end{defn}

\begin{rem}
One could formulate this using any flavour of enhancement that one had a taste for; the essential point is that the embedding should be a truncation of one preserving all higher homotopical structure.
\end{rem}

Our main result is Theorem~\ref{thm:main} which asserts that every proper connective dg algebra admits a geometric realization, and moreover one can choose the smooth projective scheme to have a full exceptional collection. We then deduce a number of strong consequences from this fact.

Our approach is via the machinery of $A_{\infty}$-algebras in the category of chain complexes. This has the advantage that we can work with minimal models of proper dg algebras which are honestly finite dimensional. We will see in Corollary~\ref{cor:fd} that in fact any proper connective dg algebra admits a finite dimensional dg model, see also Remark \ref{rem:fd}.

We first show that for a finite dimensional connective $A_{\infty}$-algebra one can construct a finite analogue of the radical filtration which is compatible with the higher multiplications. We then go on to use this, together with the results of \cite{Maaseik} to deduce geometricity.

 In this section we assume $k$ is algebraically closed.

\subsection{An $A_{\infty}$-version of the radical filtration}

Assume $(\A,m_*)$ is a strictly unital and minimal $A_{\infty}$-algebra (i.e. $m_1=0$). For $n \geq 2$, denote by $\Psi_n \subset \Hom_k(\A^{\otimes n},\A)$ the subset consisting of all $n$-ary operations obtained by successively applying $1^{\otimes r} \otimes m_s \otimes 1^{\otimes t}$ to $\A^{\otimes n}$, and set $\Psi_1:=\{\id_{\A}\}$. One can also think of $\Psi_n$ as the set of planar rooted trees with $n$ leaves. In particular each of the sets $\Psi_n$ is finite. Set $\Psi=\bigcup_{n \geq 1} \Psi_n$ and define two functions
\begin{align}
v&:\Psi \to \NN\\
\vert \cdot \vert&:\Psi \to \NN
\end{align}
as follows: if $\psi \in \Psi$ is constructed out of $m_{i_1}, \ldots, m_{i_k}$, then $v(\psi):=k$ (i.e.\ the number of higher multiplications involved, or equivalently the number of inner vertices of the planar rooted tree corresponding to $\psi$), and $\vert \psi \vert:=\vert \sum_{j=1}^k (2-i_j) \vert$ is the absolute value of the total degree of $\psi$. It is then easy to check that for $\psi \in \Psi_n$ we have
\begin{equation}
\label{eq:form}
n=v(\psi)+|\psi|+1.
\end{equation}

\begin{ex}
For $n=4$, the set $\Psi_4$ consists of 11 elements:
\begin{center}
  \begin{tabular}{ l | r | r }
    $\Psi_4$ & $v$ & $| \cdot |$ \\ \hline
    $m_4(-,-,-,-)$ & 1 & 2 \\
    $m_3(m_2(-,-),-,-)$ & 2 & 1 \\
    $m_3(-,m_2(-,-),-)$ & 2 & 1\\
    $m_3(-,-,m_2(-,-))$ & 2 & 1 \\
    $m_2(m_3(-,-,-),-)$ & 2 & 1 \\
    $m_2(-,m_3(-,-,-))$ & 2 & 1 \\
    $m_2(m_2(m_2(-,-),-),-)$ & 3 & 0 \\
    $m_2(m_2(-,m_2(-,-)),-)$ & 3 & 0 \\
    $m_2(m_2(-,-),m_2(-,-))$ & 3 & 0 \\
    $m_2(-,m_2(m_2(-,-),-))$ & 3 & 0 \\
    $m_2(-,m_2(-,m_2(-,-)))$ & 3 & 0
  \end{tabular}
\end{center}
\end{ex}
Since $\A$ is minimal, $m_2$ defines an associative multiplication on $\A$, and we can consider the corresponding Jacobson radical $J:=\rad \A$. The radical defines a decreasing filtration $F^*:=\{F^p\A\}_{p\geq 0}$ on $\Lambda$, where 
\begin{align*}
F^0\A&= \A\\
F^1\A&=\sum_{n \geq 1} \sum_{\psi \in \Psi_n} \psi(J^{\otimes n})\\
F^2\A&=\sum_{n \geq 2} \sum_{\psi \in \Psi_n} \psi(J^{\otimes n})\\
&\vdots\\
F^{k}\A&=\sum_{n \geq k} \sum_{\psi \in \Psi_n} \psi(J^{\otimes n})\\
&\vdots
\end{align*}
This filtration is not necessarily finite, even if $\A$ is finite-dimensional.
\begin{ex}
Consider the quiver algebra
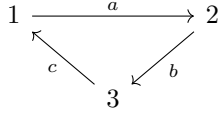
\begin{equation}
\begin{tikzcd}
1 \ar{rr}{a}& & 2 \ar{dl}{b}\\
& 3 \ar{ul}{c} &
\end{tikzcd}
\end{equation}
with $a$ in degree $1$, $b$ and $c$ in degree $0$, and with square-zero relations. Then we obtain a minimal finite-dimensional $A_{\infty}$-algebra by defining
\begin{align}
m_3(a,b,c)&=e_1 \\ 
m_3(b,c,a)&=e_2 \\ 
m_3(c,a,b)&=e_3
\end{align}
and setting $m_n=0$ for $n>3$. The filtration $F^*$ is infinite. Note that $m_3$ does not preserve $J$. This $A_{\infty}$-algebra is a minimal model for the Ext-algebra of the (sum of the) 3 indecomposable modules for the path algebra of the $A_2$ quiver.
\end{ex}

Even if the higher multiplications preserve $J$, i.e. $m_n(J^{\otimes n}) \subset J$ for all $n$, then $F^*$ is not necessarily finite, see Remark \ref{rem:efi}. The following lemmas are needed to prove the main result in \S\ref{sec:georeal}.

\begin{lem}
\label{lem:filt}
If $(\A,m_*)$ is a minimal finite-dimensional $A_{\infty}$-algebra which is connective (i.e.\ $\A^i=0$ for $i>0$), then:
\begin{enumerate}
\item\label{eq:radpres} $m_n(J^{\otimes n}) \subset J$ for all $n$;
\item \label{eq:1rad}$F^1\A=J$;
\item\label{eq:filtfin} the filtration $F^*$ is finite.
\end{enumerate}  
\end{lem}
\begin{proof}
The assumptions on $\A$ imply that
\begin{equation}
\label{desc-rad}
J=\rad(\A^0) \oplus \A^{<0},
\end{equation}
where $\A^{<0}=\bigoplus_{i<0}\A^i$. 

To show \eqref{eq:radpres} and \eqref{eq:1rad}, it suffices to note that the inclusion $m_2(J^{\otimes 2}) \subset J$ is satisfied for every associative algebra, and if $n>2$ then for degree reasons it follows from \eqref{desc-rad} that $m_n(J^{\otimes n}) \subset J$.

As for \eqref{eq:filtfin}, it suffices to show that $\psi(J^{\otimes n})=0$ for every $\psi \in \Psi_n$ with $n \gg 0$. To prove this, first note there exists an  $n_0$ such that $\A^i=0$ for $i<-n_0$, and let $l(\A)$ denote the radical length of $(\A,m_2)$, so $J^{i}=0$ for $i > l(\A)$. We claim that choosing 
\begin{equation}
\label{eq:bound-vertices}
n_1> n_0(l(\A)-1)+(l(\A)-2)
\end{equation}
ensures that $v(\psi) \leq n_1$ for all $\psi \in \Psi_n$ satisfying $\psi(J^{\otimes n})\neq 0$ (in particular, $n_1$ is independent of $n$). Indeed, since $\A^i=0$ for $i>0$, for any $\psi \in \Psi$ the maximal degree of an element in the domain of $\psi$ is $0$, and the minimal degree of an element in the codomain is $-n_0$. The $\psi$ with maximal possible $v(\psi)$ such that $\psi(J^{\otimes n})\neq 0$ is thus obtained by:
\begin{enumerate}
\item applying $m_2$ at most $l(\A)-1$ times,
\item applying $m_3$ to lower the degree by $-1$,
\item iterate.
\end{enumerate}
This process can only be iterated until we reach degree $-n_0$, which explains \eqref{eq:bound-vertices}.

For degree reasons, we can choose an $n_2$ such that $m_n=0$ for $n > n_2$ (e.g.\ $n_0+3$). To finish the proof, we now choose $n > n_1+\vert n_1(2-n_2) \vert+1$. For $0 \neq \psi \in \Psi_n$ we then have by formula \eqref{eq:form}
\begin{equation}
v(\psi)+|\psi| > n_1+\vert n_1(2-n_2) \vert,
\end{equation}
so either $v(\psi)>n_1$ or an $m_\ell$ with $\ell>n_2$ occurs; in either case this forces $\psi(J^{\otimes n})=0$.
\end{proof}

\begin{rem}
\label{rem:efi}
Finiteness of the filtration, as demonstrated in Lemma~\ref{lem:filt}, seems to be quite subtle: in \cite[Theorem 5.4]{efimov}, a minimal finite-dimensional $A_{\infty}$-algebra $(B,m_*)$ is constructed satisfying $m_n(J^{\otimes n}) \subset J$, but with elements $a,b,c \in J$ such that $m_3(a,b,c)=b$. In particular, the filtration $F^*$ is not finite. Of course, $B$ is not connective.
\end{rem}

We will say that a minimal finite-dimensional $A_{\infty}$-algebra $(\A,m_*)$ is basic if the associative algebra $(\A,m_2)$ is basic, i.e. $\A/J \cong k \times \cdots \times k$. In particular, there exists a subalgebra $S \subset \A$ spanned by primitive orthogonal idempotents such that $\A=S \oplus J$. By considering $\A$ as an $A_{\infty}$-category with one object for every such idempotent, we can hence assume that 
\begin{equation}
\label{eq:sssub}
m_n(\A,\ldots,\A,S,\A,\ldots,\A)=0,
\end{equation}
for $n>2$ (see for instance \cite{Seidel08}*{Lemma~2.1}).

\begin{lem}
\label{lem:filt2}
If $\A$ is a minimal basic connective finite-dimensional $A_{\infty}$-algebra, then the compatibility condition
\begin{equation}
\label{eq:compatibility}
m_v(F^{i_1} , \ldots , F^{i_v}) \subset F^{i_1+\cdots+i_v},
\end{equation}
is satisfied for all $v\geq 1$ and $i_1,\ldots,i_v$. 
\end{lem}
\begin{proof}
For fixed $i_1,\ldots,i_v$ we induct on the number of $i_j$ which are equal to $0$. In the base case, i.e.\ when all the $i_j\geq 1$, then 
\begin{eqnarray}
m_v(F^{i_1} , \ldots , F^{i_v})
&=& m_v\bigg(\sum_{n_1 \geq i_1} \sum_{\psi_1 \in \Psi_{n_1}} \psi_1(J^{\otimes n_1}), \ldots, \sum_{n_v \geq i_v} \sum_{\psi_v \in \Psi_{n_v}} \psi_v(J^{\otimes n_v})\bigg) \\
&=& \sum_{\substack{n_1,\ldots,n_v \\ n_j \geq i_j}} \sum_{\substack{\psi_1,\ldots,\psi_v \\ \psi_j \in \Psi_{n_j}}}  m_v(\psi_1(J^{\otimes n_1}),\ldots, \psi_v(J^{\otimes n_v})) \\
&=& \bigg(\sum_{\substack{n_1,\ldots,n_v \\ n_j \geq i_j}} \sum_{\substack{\psi_1,\ldots,\psi_v \\ \psi_j \in \Psi_{n_j}}}  m_v(\psi_1,\ldots, \psi_v)\bigg)(J^{\otimes n_1 + \cdots + n_v})\\
&\subset & F^{i_1 + \cdots + i_v}.
\end{eqnarray}
We split the induction into two cases. Suppose first that $v\geq 3$ and the compatibility statement holds when at most $k-1$ indices are $0$. Given $v\geq k$, and $i_1,\ldots,i_v$ with $k$ of them $0$, we note that, if $i_j=0$, then 
\begin{eqnarray}
m_v(F^{i_1} , \ldots , F^{i_v}) &=&
m_v(F^{i_1} , \ldots , F^{i_{j-1}} , \A , F^{j+1} , \ldots , F^{i_v}) \\
&=&m_v(F^{i_1} , \ldots , F^{i_{j-1}} , S , F^{j+1} , \ldots , F^{i_v})\\
&+&m_v(F^{i_1} , \ldots , F^{i_{j-1}} , J , F^{j+1} , \ldots , F^{i_v}) \\
&=&m_v(F^{i_1} , \ldots , F^{i_{j-1}} , J , F^{j+1} , \ldots , F^{i_v}) \\
&=&m_v(F^{i_1} , \ldots , F^{i_{j-1}} , F^1 , F^{j+1} , \ldots , F^{i_v}),
\end{eqnarray}
where the non-trivial identifications follow from our identification $\A = S\oplus J$, \eqref{eq:sssub}, and Lemma~\ref{lem:filt}\eqref{eq:1rad} respectively. By the induction hypothesis this lies in 
\begin{displaymath}
F^{i_1+\cdots+i_{j-1}+1+i_{j+1}+\cdots +i_v} \subseteq F^{i_1+\cdots+i_{j-1}+i_j+i_{j+1}+\cdots +i_v}.
\end{displaymath}

In the case $v=2$ it suffices to note that 
\begin{equation}
m_2(F^i , S) + m_2(F^i , J) \subset F^i.
\end{equation}
The first term lies in $F^i$ as the higher multiplications can be assumed to respect the decomposition into indecomposable summands, and the second term lies in $F^i$ by the base case for the induction.
\end{proof}

\subsection{Existence of realizations}
\label{sec:georeal}
In this section we show that every proper connective dg algebra admits a geometric realization in the sense of Orlov \cite{MR3545926}. We will continue to use $A_\infty$-structures, so let us phrase the notion of geometric realization in this setting.

\begin{defn}
\label{def:non-FM}
Let $\aa$ and $\bb$ be pretriangulated $A_\infty$-categories  \cite{BLM}  and put $\Ascr=H^0(\aa)$ and $\Bscr=H^0(\bb)$.
We say that an exact
functor $F\colon \Ascr\r \Bscr$ is  \emph{Fourier-Mukai} if
there is an $A_\infty$-functor $f\colon \aa\r \bb$ such that
$F\cong H^0(f)$ as graded functors. 
\end{defn}

In this language Definition~\ref{defn:gr} has the following reformulation.

\begin{defn*}
A dg algebra $A$ admits a \emph{geometric realization} if there exists a fully faithful Fourier-Mukai functor
\begin{equation}
\Perf A \to D^b(\coh(X)),
\end{equation}
for some smooth projective scheme $X$.
\end{defn*}

Let $A$ denote a proper connective dg algebra over $k$. 

\begin{lem}
\label{lem:basic}
Every proper connective dg algebra $A$ is derived Morita equivalent to a proper connective dg algebra $A'$ such that $\sfH^0(A')$ is basic.
\end{lem}
\begin{proof}
Since $A$ is proper, $\Perf A$ is Krull-Schmidt, so $A\cong \oplus_{i=1}^n P_i^{\oplus m_i} \in \Perf A$, with $P_1,\ldots,P_n$ non-isomorphic indecomposable dg modules. Since $A$ generates $\Perf A$, so does $P:=P_1 \oplus \cdots \oplus P_n$ and hence $\Perf A \cong \Perf A'$ with $A':=\mathrm{R}\End_{\Perf A}(P)$. It hence suffices to show that $\sfH^0(A')$ is basic. This follows from the fact that the functor $\sfH^0$ induces an equivalence
\begin{equation}
\sfH^0=\Hom_{\Perf A}(A,-)\colon \add(A) \xrightarrow{\sim} \proj \sfH^0(A),
\end{equation}
which follows for instance from \cite{HA}*{Corollary~7.2.2.19}.
\end{proof}

Since we are working over a field, the theory of minimal models ensures there exists a minimal $A_{\infty}$-structure on the finite-dimensional vector space $\A:=\sfH^*(A')$ such that the resulting $A_{\infty}$-algebra $(\A,m_*)$ satisfies
\begin{equation}
\label{eq:eq-ainf}
\Perf A \cong \Perf A' \cong \Perf \A,
\end{equation}
and these equivalences are Fourier-Mukai. In particular, $\A$ is a minimal basic connective finite dimensional $A_{\infty}$-algebra.

\begin{thm}
\label{thm:main}
Every proper connective dg algebra $A$ admits a geometric realization
\begin{equation}
\Perf A \to D^b(\coh(X)),
\end{equation}
for some smooth projective scheme $X$ with a full exceptional collection.
\end{thm}
\begin{proof}
By the discussion before the theorem, Lemmas \ref{lem:filt} and \ref{lem:filt2} apply to $\A$. This allows us to construct the corresponding Auslander $A_{\infty}$-category $\mathsf{\Gamma}=\mathsf{\Gamma}_{\A,F^*}$ defined in \cite[\S3]{Maaseik}. By applying Corollary 3.2 and Proposition 3.3 in loc.\ cit., there is a fully faithful Fourier-Mukai functor 
\begin{equation}
\label{eq:emb-aus}
\mathsf{\Gamma} \overset{\infty}{\otimes}_{\A} -\colon \Perf \A \to \Perf \mathsf{\Gamma},
\end{equation}
and $\mathsf{\Gamma}$ admits a (finite) semi-orthogonal decomposition
\begin{align}
\Perf \mathsf{\Gamma}&=\langle \Perf \A/F^1 ,\ldots,\Perf \A/F^1 \rangle \\
&=\langle \Perf S,\ldots,\Perf S \rangle
\end{align}
so in particular $\Perf \mathsf{\Gamma}$ admits a full exceptional collection. Applying \cite[Theorem 5.8]{MR3545926}, there exists a fully faithful Fourier-Mukai functor 
\begin{equation}
\label{eq:emb-scheme}
\Perf \mathsf{\Gamma} \to D^b(\coh(X)),
\end{equation}
for some smooth projective scheme $X$ with a full exceptional collection. It now suffices to consider the composition of \eqref{eq:emb-aus} and \eqref{eq:emb-scheme} with the equivalence \eqref{eq:eq-ainf} to finish the proof of the theorem.
\end{proof}

\begin{rem}
See Proposition~\ref{thm:geometric} for an alternative argument in the special case that $\sfH^0(A)$ is smooth.
\end{rem}


It follows that every proper connective dg algebra $A$ is derived Morita equivalent to a finite dimensional one. In fact, one can even deduce that $A$ is quasi-isomorphic to a finite dimensional dg algebra. We are grateful to Dmitri Orlov for pointing out to us that this second stronger conclusion follows from his work.

\begin{cor}
\label{cor:fd}
Every proper connective dg algebra $A$ is quasi-isomorphic to a finite-dimensional dg algebra.
\end{cor}
\begin{proof}
By combining Theorem \ref{thm:main} with \cite[Corollary 2.20]{orlov2019finite} we see that any proper connective dg algebra $A$ is derived Morita equivalent to a finite dimensional one, say $B$. That is, there is a quasi-equivalence of dg categories of perfect complexes
\begin{displaymath}
\mcF \colon \PERF A \to \PERF B,
\end{displaymath} 
with associated quasi-isomorphism $A \stackrel{\sim}{\to} \PERF B(\mcF A, \mcF A)$. By \cite{orlov2019finite}*{Proposition~2.5} the dg algebra $\PERF B(\mcF A, \mcF A)$ is quasi-isomorphic to a finite dimensional one.
\end{proof}

\begin{rem}
\label{rem:fd}
One can also give a direct proof of Corollary \ref{cor:fd}, which we sketch here. First replace $A$ by a minimal $A_{\infty}$-model $\aa$. Now, by \cite[Lemma 7.4.0.1]{lefevre2002categories}, there is a quasi-fully faithful $A_{\infty}$-Yoneda functor $\aa\r \Modu\aa$ with target the dg category of right $A_\infty$ $\aa$-modules. The explicit description of the hom-complexes in this dg category \cite{MR2600992} then allows one to show that the dg algebra $\underline{\hom}_{\Modu\aa}(\aa,\aa)$ (which is $A_{\infty}$-quasi-isomorphic to $\aa$) is locally finite and bounded above. By truncating in sufficiently negative degrees to avoid cohomology, we obtain a finite dimensional dg model for $A$.
\end{rem}

As another corollary, we obtain the following answer to \cite[Question 4.4]{MR3545926} in the connective case.

\begin{cor}
Every smooth proper noncommutative scheme that is connective is geometric.
\end{cor}
\begin{proof}
This follows from Theorem \ref{thm:main} combined with \cite[Proposition 3.17]{MR3545926}. 
\end{proof}



\section{An isomorphism of noncommutative Chow motives}

The geometricity result of the previous section has strong implications for the noncommutative motive of a smooth proper connective dg algebra. We will briefly recall some details of the theory and then provide two proofs that all additive invariants of a smooth proper connective dg algebra $A$ are determined by $\sfH^0(A)/\rad(\sfH^0(A))$: the first is a consequence of Theorem~\ref{thm:main} and the latter is a direct argument inspired by work of Tabuada-Van den Bergh \cite[Theorem 3.15]{MR3315059} and Keller \cite[\S 2.5]{MR1492902}.


\subsection{Recollections and the geometric proof}

Let us begin by briefly recapitulating some of the relevant details on noncommutative motives. For more details and references we refer the reader to the monograph \cite{MR3379910}. Denote by $\mathrm{dgcat}(k)$ the category of small dg categories and dg functors.

\begin{defn}
A functor $E\colon \mathrm{dgcat}(k) \to \AC$, with values in an additive category $\AC$ is called an \textit{additive invariant} if:
\begin{enumerate}
\item it sends Morita equivalences to isomorphisms,
\item given small dg categories $\mcA, \mcB$ and a dg $\mcB$-$\mcA$-bimodule $M$, the morphism 
$$E(\mcA) \oplus E(\mcB) \to E(\mcA \oplus_M \mcB)$$
induced by the obvious inclusion functors is an isomorphism, where $\mcA \oplus_M \mcB$ is the glued (aka upper triangular matrix) dg category.
\end{enumerate}
\end{defn}

In \cite{MR2196100}, a category $\mathrm{Hmo}_0(k)$ along with a functor
\begin{equation}
\label{eq:univ}
U(-)\colon \mathrm{dgcat}(k) \to \mathrm{Hmo}_0(k)
\end{equation}
were constructed, such that for any additive category $\AC$ the pushforward
\begin{equation}
U(-)_*\colon \mathrm{Fun}(\mathrm{Hmo}_0(k),\AC) \to \mathrm{Add}(\mathrm{dgcat}(k),\AC)
\end{equation}
induces an equivalence of categories, where on the left hand side we consider the category of additive functors from $\mathrm{Hmo}_0(k)$ to $\AC$, and on the right hand side we consider the category of additive invariants. Hence it makes sense to call \eqref{eq:univ} the \textit{universal additive invariant}. For this reason, $\mathrm{Hmo}_0(k)$ is called the category of \textit{noncommutative motives.}

\begin{defn}
The category of \textit{noncommutative Chow motives} $\mathrm{NChow}(k)$ is the idempotent completion of the full subcategory of $\mathrm{Hmo}_0(k)$ consisting of the objects $U(\mathcal{A})$ with $\mathcal{A}$ a smooth and proper dg category.
\end{defn}

From Theorem~\ref{thm:main} we deduce the following structural result.

\begin{cor}\label{cor:motive}
Let $k$ be an algebraically closed field and let $A$ be a smooth proper connective dg algebra over $k$. The noncommutative motive of $A$ is isomorphic to a direct sum of copies of the noncommutative motive of $k$.
\end{cor}
\begin{proof}
By the theorem we have an embedding
\begin{equation*}
\Perf A \to \sfD^\b(\coh(X)),
\end{equation*}
for some smooth projective scheme $X$ with a full exceptional collection. The existence of the full exceptional collection implies that $U(\sfD^\b(\coh(X))) \cong U(k)^{\oplus n}$. The noncommutative motive of an admissible subcategory is a summand of that of the ambient category, so it is sufficient to show that $\Perf A$ is admissible.

By Proposition~\ref{prop:regular} $\Perf A$ is regular and admissibility follows (cf.\ \cite[Proposition 3.17]{MR3545926} for instance).
\end{proof}

One can relax the assumption that $k$ is algebraically closed in the Corollary. This is done in the next section\textemdash{}see Theorem~\ref{th:motives} for a precise statement.


\subsection{A direct proof}

The first thing we will need is that the Grothendieck group of a proper connective dg algebra $A$ is finite free. This can be viewed as a generalization of the corresponding fact for finite dimensional algebras; in fact, if $A$ is smooth then it follows from the statement for finite dimensional algebras via the standard t-structure with heart $\modu \sfH^0(A)$.

In this section we work over a perfect field $k$.

\begin{lem}
\label{lem:K0}
If $A$ is proper, then the Grothendieck group $K_0(A)$ is free abelian on the indecomposable summands of $A$ in $\Perf A$.
\end{lem}
\begin{proof}
This is a particular case of the more general \cite[Theorem 2.27]{MR2927802}.
\end{proof}

\begin{lem}
\label{prop:K0}
If $A$ is proper, and $I$ is a dg ideal of $A$ such that $\sfH^0(I)$ is nilpotent in $\sfH^0(A)$, then the quotient map $\pi\colon A \to A/I$ induces an isomorphism
\begin{equation}
K_0(A) \xrightarrow{\cong} K_0(A/I).
\end{equation}
\end{lem}
\begin{proof}
Lemma \ref{lem:K0} applies both to $A$ and $A/I$. Moreover, since $\sfH^0(I)$ is nilpotent, $\sfH^0(I) \subset \rad(\sfH^0(A))$. The idempotents giving the indecomposable summands of $A$ in $\Perf A$ are lifted from $\sfH^0(A)/\rad(\sfH^0(A))$. This also holds for $A/I$ by nilpotence of $\sfH^0(I)$. In particular, the functor
\begin{equation}
\pi^*\colon \Perf A \to \Perf A/I
\end{equation}
induced by $\pi$ induces a bijection between the indecomposable summands of $A$ and those of $A/I$, and hence induces an isomorphism on the level of $K_0$.
\end{proof}

With these results in hand we now prove the theorem. 

\begin{thm}
\label{th:motives}
If $A$ is smooth, proper, and connective, then the morphism \\ $\pi\colon A \to \sfH^0(A)/\rad(\sfH^0(A))$ induces an isomorphism
\begin{equation}
\label{eq:chow}
U(\pi)\colon U(A) \xrightarrow{\sim} U(\sfH^0(A)/\rad(\sfH^0(A)))
\end{equation} 
in the category $\mathrm{NChow}(k)$ of noncommutative Chow motives.
\end{thm}
\begin{proof}
Set $S:=\sfH^0(A)/\rad(\sfH^0(A))$. First note that $S$ is smooth and proper, since $k$ being perfect ensures $S$ is separable. By assumption $A$ is smooth and proper as well, so $U(A),U(S) \in \mathrm{NChow}(k)$. By the Yoneda lemma, to show \eqref{eq:chow} it suffices to show that 
\begin{equation}
U(\pi)_*\colon \Hom_{\mathrm{NChow}(k)}(U(\Lambda),U(A)) \to \Hom_{\mathrm{NChow}(k)}(U(\Lambda),U(S))
\end{equation}
is an isomorphism, for $\Lambda$ equal to $A$ and to $S$. By \cite[\S 4.1]{MR3379910}, this reduces to showing that 
\begin{equation}
\label{eq:K0'}
K_0(\Lambda^{\op} \otimes_k A) \to K_0(\Lambda^{\op} \otimes_k S)
\end{equation}
is an isomorphism for both values of $\Lambda$. 

Note that $\Lambda^{\op} \otimes_k A$ and $\Lambda^{\op} \otimes_k S$ are still proper and connective. Let $I$ denote the dg ideal $\ker(\pi\colon A\to S)$. We have an exact sequence of complexes
\begin{equation}
0 \to \Lambda^{\op}\otimes_k I \to \Lambda^{\op}\otimes_k A \to \Lambda^{\op}\otimes_k S \to 0,
\end{equation}
where the surjection is a dg algebra map inducing a surjective map with nilpotent kernel on $\sfH^0$. Thus $\Lambda^{\op}\otimes_k I$ is a dg ideal with nilpotent cohomology in degree $0$ and Lemma \ref{prop:K0} applies to yield that \eqref{eq:K0'} is an isomorphism.
\end{proof}

If we further specialize to $k$ being algebraically closed, then we obtain an isomorphism 
\begin{equation}
U(A) \xrightarrow{\cong}U(k)^{\oplus |\sfH^0(A)|}
\end{equation}
where $|\sfH^0(A)|$ is the number of non-isomorphic simple $\sfH^0(A)$-modules, and hence the noncommutative motive of $A$ is a sum of copies of the noncommutative motive of $k$, i.e. it is of \textit{unit type} \cite[\S 4]{MR3090263}, in agreement with Corollary~\ref{cor:motive}.

\begin{rem}
Note that we have opted to work with smooth and proper dg algebras instead of smooth and proper dg categories. This is no actual restriction, since any smooth and proper dg category has a classical generator, and is thus derived Morita equivalent to a smooth and proper dg algebra. (See \cite{MR3379910}*{Proposition~1.45} for a proof.)
\end{rem}


\subsection{Examples}

Let us give some simple examples dispelling some naive beliefs one could hold regarding the structure of smooth proper connective dg algebras.

\begin{ex}
Consider the graded Kronecker quiver
\begin{equation}
\begin{tikzcd}
1 \ar[shift left]{r}{a} \ar[shift right,swap,dashed]{r}{b} & 2,
\end{tikzcd}
\end{equation}
with $|a|=0$ and $|b|=-2$, as a dg algebra $A$ with trivial differential. Then $A$ is smooth, proper and connective, but $A$ is not derived Morita equivalent to a finite dimensional algebra. Indeed, one can calculate (directly, using that the algebra is quite small, or see \cite[Proposition 3.1]{MR3441111}) that
\begin{equation}
\sfH \sfH^{-1}(A)=k,
\end{equation}
and (dg algebras derived Morita equivalent to) finite dimensional algebras do not have negative Hochschild cohomology. 
\end{ex}

\begin{ex}
Consider the minimal $A_{\infty}$-algebra $A$ which has underlying quiver
\begin{equation}
\begin{tikzcd}
1 \ar{r}{a} \ar[bend right,dashed]{rrr}{e} & 2 \ar{r}{b} & 3 \ar{r}{c} & 4
\end{tikzcd}
\end{equation}
such that $|a|=|b|=|c|=0, |e|=-1$ and with higher multiplications determined by $m_1=0$, $m_2(a,b)=m_2(b,c)=0$, and $m_3(a,b,c)=e$. Then $A$ is smooth, proper and connective, but $A$ is not formal: to see this remember that a minimal model for $A$ is unique up to (non-unique) $A_{\infty}$-isomorphism \cite[Corollaire 1.4.1.4]{lefevre2002categories}, and one checks explicitly that there are no $A_{\infty}$-morphisms from $\sf{H}^*(A)$ to $A$. 
\end{ex}

Both of these examples are directed and so, as can be readily checked, admit full exceptional collections. This gives another computation of the noncommutative Chow motive (and a verification for the skeptical reader that the examples are really smooth).

One can also produce examples of smooth and proper connective dg algebras which do not admit a full exceptional collection, for instance via silting mutation \cite{MR2927802} from a smooth finite dimensional algebra with no full exceptional collection (see \cite{happel1991family}).

\begin{ex}
We give an explicit example. Consider the algebra $A_3$ described by the quiver 
\[
\begin{tikzcd}
1 \ar[bend left]{rr} \ar[bend right]{rr} && 2 \ar{ll}
\end{tikzcd}
\]
with arrows labelled $a_i\colon 1 \to 2$ for $i=1,2$ and $b\colon 2\to 1$, and relations $ba_1$ and $a_2b$. This is an $8$ dimensional algebra of finite global dimension (in fact it is smooth), and the proof of the first proposition in Section~3 of \cite{happel1991family} shows that $\Perf A_3$ admits no exceptional objects. 

One can then produce further examples by taking derived endomorphism rings of silting complexes. For instance, consider the triangle in $\Perf A_3$
\[
\begin{tikzcd}
P_2 \ar{r}{b} & P_1 \ar{r} & X \ar{r} & \Sigma P_2.
\end{tikzcd}
\]
By construction $P_1 \oplus X$ is a generator and it is easily checked to be silting but not tilting. We leave the computation of the corresponding dg algebra to the interested reader.
\end{ex}

We suspect there are a plethora of such examples which are, moreover, not derived Morita equivalent to an algebra. 



\section{Smoothness, regularity and t-structures}
\label{sec:appendix}

This section is devoted to recording a few technical results concerning the relationship between smoothness and regularity for dg algebras, which are particularly useful in the connective case.


\subsection{Smoothness implies regularity}\label{ssec:sir}

In this section assume $A$ is a dg algebra (not necessarily connective unless otherwise stated) over a commutative base ring $\mathbbm{k}$. The advantage of smooth dg algebras is that one has access to strong representability statements. The rub is that the representability statements are generally phrased in terms of another condition: we say our dg algebra $A$ is \emph{regular} if $\Perf A$ is strongly generated, i.e.\ there is a uniform bound on the number of cones required to build any object from $A$. We will show smoothness implies regularity in the cases where this is reasonable. When $\mathbbm{k}$ is a field this has been proved by Lunts \cite[Lemmas 3.5, 3.6]{lunts2010categorical} and the argument we give here is based upon his.

For an object $M\in \sfD(A)$ we set
\begin{equation}
\thick_1(M) = \add(\Sigma^iM \mid i\in \ZZ) \text{ and } \loc_1(M) = \Add(\Sigma^iM \mid i\in \ZZ),
\end{equation}
i.e. we start with $M$ and close under suspensions, sums, and summands (allowing infinite sums in the second case). We then inductively define further full subcategories in the usual way, e.g.\
\begin{equation}
\thick_n(M) = \add(N \mid \exists \text{ a triangle } N_{n-1} \to N \to N_1 \text{ with } N_i\in \thick_{i}(M)).
\end{equation}
When it is necessary to distinguish the ambient category we use a superscript to indicate the relevant dg algebra, for instance $\thick^\mathbbm{k}_n(M)$ versus $\thick^A_n(M)$. 

\begin{lem}\label{lem:2.5}
If $A\in \thick_d^{A^\mathrm{e}}(A^\mathrm{e})$ and $M\in \Perf A$ lies in $\loc^\mathbbm{k}_n(\mathbbm{k})$, then $M$ lies in $\thick^A_{n+d}(A)$.
\end{lem}
\begin{proof}
The base change functor
\begin{equation}
\begin{tikzcd}
\sfD(\mathbbm{k})  \ar{rr}{-\otimes_\mathbbm{k} A} && \sfD(A)
\end{tikzcd}
\end{equation}
is exact and coproduct preserving so from $M\in \loc^\mathbbm{k}_n(\mathbbm{k})$ we deduce that $M\otimes_\mathbbm{k} A \in \loc^A_n(A)$. Now note that
\begin{equation}
M\otimes_\mathbbm{k} A \cong M\otimes_A A^\op\otimes_\mathbbm{k} A \text{ and } M\cong M \otimes_A A.
\end{equation}
Thus $A\in \thick_d^{A^\mathrm{e}}(A^\mathrm{e})$ implies that $M \in \thick^A_d(M\otimes_\mathbbm{k} A)$. Combining what we have learned we see that
\begin{equation}
M \in \loc^A_{n+d}(A)
\end{equation}
and perfection of $M$ implies, via \cite[Proposition~2.2.4]{bondal2003generators}, that $M\in \thick^A_{n+d}(A)$ as claimed.
\end{proof}

\begin{prop}\label{prop:regular}
If $A$ is smooth and $\mathbbm{k}$ has finite global dimension then $A$ is regular.
\end{prop}
\begin{proof}
Since $A$ is smooth we have $A\in \thick^{A^\mathrm{e}}(A^\mathrm{e})$ and so there exists some $d\geq 1$ such that $A\in \thick^{A^\mathrm{e}}_d(A^\mathrm{e})$. By \cite{neeman2017strong}*{Theorem~2.1} there exists a perfect complex $G$ over $\mathbbm{k}$ with $\loc^\mathbbm{k}_m(G) = \sfD(\mathbbm{k})$. As $G$ is perfect there is an $m'$ with $G\in \thick^\mathbbm{k}_{m'}(\mathbbm{k})$ and it follows that there is an $n\leq m+m'$ with $\loc^\mathbbm{k}_n(\mathbbm{k}) = \sfD(\mathbbm{k})$. We then see from Lemma~\ref{lem:2.5} that for each $M\in \Perf A$ we have $M\in \thick^A_{n+d}(A)$ i.e.\
\begin{equation}
\Perf A = \thick^A_{n+d}(A).
\end{equation}
Hence $\Perf A$ is strongly generated and $A$ is regular.
\end{proof}

\begin{rem}
If $\mathbbm{k}$ does not have finite global dimension then, even though we allow infinite coproducts, one doesn't get the required bound to apply Lemma~\ref{lem:2.5} (cf.\ \cite{rouquier08}*{Lemma~7.13}). This is not terribly surprising given the situation for honest rings: smooth algebras over $\mathbbm{k}$ won't in general be any better behaved than $\mathbbm{k}$.
\end{rem}

From this we can give a description of the perfect complexes over certain smooth connective dg algebras and show that the standard t-structure restricts to perfect complexes.

\begin{prop}\label{prop:t}
Suppose that $\mathbbm{k}$ has finite global dimension. Let $A$ be smooth connective and further assume that $\sfH^*(A)$ is coherent and $\sfH^{\leq -1}(A)$ is a finitely generated ideal. Then, there is an identification 
\begin{equation}
\Perf A = \{M\in \sfD(A) \mid \sfH^*(M) \text{ is finitely presented over } \sfH^*(A)\} =: \sfD^\mathrm{fp}(A).
\end{equation}
Moreover, if $\sfH^*(A)$ is noetherian, then the standard t-structure on $\sfD(A)$ restricts to $\Perf A$.
\end{prop}
\begin{proof}
It is clear that $\Perf A \subseteq \sfD^\mathrm{fp}(A)$. If, on the other hand $M \in \sfD^\mathrm{fp}(A)$ then we can consider the cohomological functor
\begin{equation}
H=\Hom_{\sfD(A)}(-,M)\colon \Perf A^\op \to \Modu\mathbbm{k}.
\end{equation} 
By assumption the cohomology of $M$ is finitely presented over $\sfH^*(A)$ and so,  by \cite{GreenleesStevenson17}*{Proposition~4.6} the functor $H$ is locally finitely presented in the sense of \cite{rouquier08}*{Section~4.1.1}. By Proposition~\ref{prop:regular} we know that $A$ is regular. Hence, we can apply \cite{rouquier08}*{Theorem~4.16} to deduce that $H$ is represented by some $M' \in \Perf A$: 
\begin{equation}
H \cong \Hom_{\Perf A}(-,M').
\end{equation}
Using $\id_{M'}$, we obtain a universal morphism $f\colon M' \to M$
in $\sf{D}(A)$ which is, by the fact that $M'$ represents $H$, an isomorphism after applying $\Hom_{\sfD(A)}(\Sigma^iA,-)$ for any $i\in \ZZ$. In other words, $f$ is a quasi-isomorphism and we conclude $M \cong M'$ is perfect.

It remains to check that the t-structure restricts to $\sfD^\mathrm{fp}(A)$ when the cohomology is noetherian. This is immediate from the fact that the truncation of a finitely presented graded $\sfH^*(A)$-module is again finitely presented.
\end{proof}

The same method of proof works to show the following proposition. Denote by $B=\sfH^0(A)$ and consider the adjunction 
\begin{equation}
\begin{tikzcd}
\pi^*=-\otimes_{A} B:\sf{D}(A) \ar[shift left]{r} & \sf{D}(B):-\otimes_{B} B_{A}=\pi_* \ar[shift left]{l}
\end{tikzcd}
\end{equation}
where $\sf{D}^\mathrm{b}(\modu B)$ denotes the bounded derived category of finitely presented $B$-modules. 

\begin{prop}
\label{lem:3}
Suppose that $\mathbbm{k}$ has finite global dimension. Let $A$ be smooth connective and further assume that $\sfH^*(A)$ is coherent and $\sfH^{\leq -1}(A)$ is a finitely generated ideal. Then for any $X \in \sf{D}^\mathrm{b}(\modu B)$ the restriction $\pi_*X$ is perfect, i.e.\ $\pi_*X \in \Perf A$. In particular, we have
 \begin{equation}
 \pi^*\pi_*X \in \Perf B.
 \end{equation}
\end{prop}
\begin{proof}
Since $\pi^*(\Perf A) \subset \Perf B$, it suffices to prove that for $X \in \sf{D}^\mathrm{b}(\modu B)$ the restriction $\pi_*X$ lies in $\Perf A$. Consider the cohomological functor 
\begin{equation}
H=\Hom_{\sfD(A)}(-,\pi_*X)\colon \Perf A^\op \to \Modu\mathbbm{k}.
\end{equation} 
Since $X$ has finitely presented cohomology over $B$ and, by assumption, $B$ is finitely presented over $\sfH^*(A)$ it follows that the cohomology of $\pi_*X$ is finitely presented over $\sfH^*(A)$. Proceeding exactly as in the proof of Proposition \ref{prop:t}, we find that $\pi_*X$ is perfect.
\end{proof} 


\subsection{From regularity to smoothness}

For a commutative local ring one can check regularity on the residue field, and for a finite dimensional algebra one can check the global dimension by examining the simples. If the base field is perfect (and under suitable finiteness conditions) one can then conclude smoothness.

The situation seems more delicate in the derived setting. We provide here a small step in this direction, which can be applied to our main objects of study: proper connective dg algebras. We return to working over a field $k$.

\begin{lem}\label{lem:Serre}
Let $A \stackrel{\pi}{\to} S$ be an augmented proper dg algebra such that:
\begin{itemize}
\item[(1)] $S\in \Perf A$;
\item[(2)] $A\in \thick_{A^e}(S^e)$.
\end{itemize}
Then $A$ is smooth over $k$.
\end{lem}
\begin{proof}
By assumption $S^\op\in \Perf A^\op$, i.e.\ $S^\op \in \thick^{A^\op}(A^\op)$ and so $S^\op\otimes_kS \in \thick^{A^e}(A^\op\otimes_k S)$. The base change functor
\begin{displaymath}
A^\op\otimes_k-\colon \sfD(A) \to \sfD(A^e)
\end{displaymath}
preserves compacts and so $A^\op\otimes_k S$ is perfect over $A^e$. Thus $S^e = S^\op\otimes_k S$ is perfect over $A^e$. Assumption (2) guarantees that $S^e$ builds $A$ over $A^e$ and so $A\in \Perf A^e$, i.e.\ $A$ is smooth.
\end{proof}

In the connective case (2) is easily checked.

\begin{lem}
Let $A$ be a proper connective dg algebra and set $S = \sfH^0(A)/\rad(\sfH^0(A))$. Suppose that $S$ is separable over $k$. Then in $\sfD(A^e)$ we have
\begin{displaymath}
A \in \thick_{A^e}(S^e).
\end{displaymath}
\end{lem}
\begin{proof}

As $S$ is separable over $k$ we have
\[
\rad \sfH^0(A^e) = \rad (\sfH^0(A))^e = \rad \sfH^0(A^\op) \otimes_k \rad \sfH^0(A)
\]
and hence
\[
\sfH^0(A^e)/\rad \sfH^0(A^e) \cong S^e.
\]
In other words $S^e$ is a sum of the simple $\sfH^0(A^e)$-modules (with each simple occurring) and so generates $\sfD^\mathrm{b}(H^0(A^e))$. It follows, using the standard t-structure, that $S^e$ generates $\sfD^\mathrm{b}(A^e)$, the category of dg $A^e$-modules with finite dimensional cohomology (cf.\ Proposition~\ref{prop:generator}). Since $A$ is proper, we have $A\in \sfD^\mathrm{b}(A^e)$ and we are done.
%
\end{proof}

\begin{thm}\label{thm:criterion}
Let $A$ be a proper connective dg algebra such that the radical quotient $S = \sfH^0(A)/\rad(\sfH^0(A))$ is separable over $k$. Then $A$ is smooth if and only if $S\in \Perf A$.
\end{thm}
\begin{proof}
By the previous lemma condition (2) of Lemma~\ref{lem:Serre} is satisfied. Thus provided (1) holds, i.e.\ $S$ is perfect over $A$, we have that $A$ is smooth.

On the other hand, if $A$ is smooth then $S$ is perfect by Proposition~\ref{lem:3}.
\end{proof}


\subsection{Regularity of the dual}

We conclude the section with a couple of comments on $\sfD^\b(A)$ for a proper connective dg algebra $A$. We begin by showing that $\sfD^\b(A)$ is always regular.

We let $A$ be a finite dimensional dg algebra and denote by $J$ the Jacobson radical of the underlying graded algebra. We set $J_- = \{a\in J\mid d(a)\in J\}$ as in \cite{orlov2019finite}. This is a two-sided nilpotent dg ideal of $A$. We set $J_+ = J + d(J)$, which is also a dg ideal such that $A/J_+$ is semisimple as a graded algebra. By \cite{orlov2019finite}*{Lemma~2.4} the natural map $A/J_- \to A/J_+$ is a quasi-isomorphism of dg algebras.

By $\sfD^\b(A)$ we mean the derived category of dg-$A$-modules with finite dimensional cohomology (or equivalently, the derived category of finite dimensional dg modules).

\begin{lem}\label{lem:Dbreg}
Let $A$ be a finite dimensional dg algebra (not necessarily connective) over $k$. Then $\sfD^\b(A)$ is regular.
\end{lem}
\begin{proof}
Let $n$ be the least integer such that $J_-^n =0$ but $J_-^{n-1}\neq 0$. Given any finite dimensional dg module $X$ consider the filtration (in the abelian category of dg modules)
\begin{displaymath}
0 = J_-^{n}X \subseteq J_-^{n-1}X \subseteq J_-^{n-2}X \subseteq \cdots \subseteq J_-X \subseteq X.
\end{displaymath}
There are short exact sequences
\begin{displaymath}
0 \to J_-^{i}X \to J_-^{i-1}X \to F_{i-1} \to 0
\end{displaymath}
where $F_{i-1}$ is naturally an $A/J_-$-module. In particular, since $A/J_- \cong A/J_+$  and the latter is semisimple as an algebra, we deduce from \cite{orlov2019finite}*{Proposition~2.16} that 
\begin{displaymath}
F_{i-1} \in \add(\Sigma^i A/J_- \mid i\in \ZZ).
\end{displaymath}
Each short exact sequence of dg modules gives a triangle in $\sfD(A)$ and so the filtration we have constructed exhibits $X$ as an object of $\thick_n(A/J_-)$ (cf.\ the start of Section~\ref{ssec:sir} for this notation).
\end{proof}

Let $k\to \ell$ be a field extension. Because being finite dimensional is preserved by base change so is the regularity we have just observed.

\begin{prop}
Let $A$ be a finite dimensional dg algebra. Given a field extension $k\to \ell$ let $A_\ell$ denote the base change $A\otimes_k \ell$ of $A$ to $\ell$. Then $\sfD^\b(A_\ell)$ is regular, i.e. $\sfD^\b(A)$ is ``geometrically regular''.
\end{prop}
\begin{proof}
Clearly $A_\ell$ is finite dimensional over $\ell$ so Lemma~\ref{lem:Dbreg} applies.
\end{proof}

Let $A$ be a finite dimensional dg algebra. We always get an augmentation $\pi\colon A\to S_+ = A/J_+$. We denote by $\DA$ the dual dg algebra $\RHom_A(S_+, S_+)$ and note that $\Perf \DA \cong \sfD^\b(A)$ so $\DA$ is regular. We remark that $\DA$ is not proper if $A$ is not smooth and it is coconnective. By Lemma~\ref{lem:Dbreg} it is regular.

\begin{lem}\label{lem:ereg}
If $S_+$ is separable over $k$ then $\DA^e$ is regular.
\end{lem}
\begin{proof}
There are identifications
\begin{align*}
\DA^e &= \RHom_A(S_+, S_+) \otimes_k \RHom_{A^\op}(S_+^\op, S_+^\op) \\
&\cong \RHom_{A^e}(S_+ \otimes_k S_+^\op, S_+\otimes_k S_+^\op) \\
&\cong \RHom_{A^e}(A^e/\rad(A^e)_+, A^e/\rad(A^e)_+)
\end{align*}
where the first isomorphism is the K\"unneth formula and the final one follows from separability.

Now $A^e$ is still finite dimensional so by Lemma~\ref{lem:Dbreg} $\sfD^\b(A^e)$ is regular. This category is derived Morita equivalent to $\RHom_{A^e}(A^e/\rad(A^e)_+, A^e/\rad(A^e)_+)$ proving the claim.
\end{proof}

Thus one could use representability techniques to prove that $\DA$ is smooth by proving that $\sfH^*(\DA)$ is sufficiently small as a $\sfH^*(\DA)^e$-module. However, it seems preferable to prove smoothness via another route and deduce this fact.



\section{Through the looking glass}
\label{sec:looking}
In this section we take a departure, at least literally, from the world of connective dg algebras; our dg algebras are not assumed to be connective. 

A dg algebra $\Gamma$ is called \textit{coconnective} if $\sfH^i(\Gamma)=0$ for $i<0$. A coconnective dg algebra $\Gamma$ is said to be \textit{simply connected} if $\sfH^1(\Gamma) = 0$. It is standard in many situations that there is a duality between connective dg algebras and coconnective simply connected dg algebras \cite{avramov1986through}. Thus using Theorem \ref{th:motives} we can also compute the noncommutative motive of many coconnective and simply connected dg algebras. We begin with some preparation in the setting we've worked in thusfar. 

Let $A$ be a proper connective dg algebra, $B=\sfH^0(A)$, and $\pi:A \to B$ the quotient morphism. Then we have an adjunction
\begin{equation}
\begin{tikzcd}
\pi^*=-\otimes_{A} B:\sf{D}(A) \ar[shift left]{r} & \sf{D}(B):-\otimes_{B} B_{A}=\pi_* \ar[shift left]{l}
\end{tikzcd}
\end{equation}

\begin{prop}\label{prop:generator}
The smallest localizing subcategory of $\sf{D}(A)$ containing $\pi_*B$ coincides with $\sf{D}(A)$:
\begin{equation}
\label{eq:loc}
\loc^{A}(\pi_*B)=\sf{D}(A).
\end{equation}
In other words, $\pi_*B$ generates $\sf{D}(A)$.
\end{prop}
\begin{proof}
Since $B$ is a (compact) generator for $\sf{D}(B)$, and the $\sfH^i(A)$ are $B$-modules, we have
\begin{equation}
\label{eq:hi}
\sfH^i(A) \in \loc^{B}(B) = \sf{D}(B) \text{ for each } i\in \ZZ.
\end{equation}
Since $\pi_*$ is exact and coproduct preserving it follows that
\begin{equation}
\pi_*\sfH^i(A) \in \loc^{A}(\pi_*B) \text{ for each } i\in \ZZ.
\end{equation}
As $A$ is proper, there is a maximal $n$ such that $\sfH^{-n}(A)\neq 0$, and it follows that the morphism
\begin{equation}
\tau^{\leq -n}A \to \pi_*\sfH^{-n}(A)
\end{equation}
in $\sf{D}(A)$ is a quasi-isomorphism. By induction, we can then build $\tau^{\leq -i}A$ using $\tau^{\leq -i-1}A$ and $\pi_*\sfH^{-i}(A)$. As $A$ is also bounded above, this shows that one can build $A$ from the $\pi_*\sfH^{-i}(A)$, and by \eqref{eq:hi} we hence obtain
\begin{equation}
A \in \loc^{A}(\pi_*B).
\end{equation}
Since $\loc^{A}(A)=\sf{D}(A)$, we obtain \eqref{eq:loc}.
\end{proof}

\begin{cor}\label{cor:cg}
Let $A$ be as above and assume in addition that $A$ is smooth. Then $\pi_*B$ is a compact generator for $\sf{D}(A)$. In particular, 
\begin{equation}
\thick^{A}(\pi_*B) = \Perf A.
\end{equation}
\end{cor}
\begin{proof}
We have just seen that $\pi_*B$ is a generator for $\sf{D}(A)$. It follows from Proposition~\ref{lem:3} that $\pi_*B$ is perfect over $A$. Thus $\pi_*B$ is a compact generator for $\sf{D}(A)$. The final statement is then a consequence of Thomason's localization theorem \cite[Theorem 2.1.3]{NeeGrot}. 
\end{proof}

\begin{cor}
\label{cor:coco}
With notation and hypotheses as above, i.e.\ $A$ is smooth, proper, and connective, then $\Perf A$ is equivalent to $\Perf \Gamma$, where 
\begin{equation}
\Gamma=\RHom_{\Perf A}(\pi_*B,\pi_*B)
\end{equation}
is a simply connected, coconnective dg algebra, with 
\begin{equation}
\sfH^0(\Gamma) = \sfH^0(A) = B.
\end{equation}
\end{cor}
\begin{proof}
By Corollary~\ref{cor:cg} $\pi_*B$ is a compact generator for $\sf{D}(A)$ and so the equivalence between $\Perf A$ and $\Perf \Gamma$ is immediate; it remains to verify that $\Gamma$ has the claimed properties. However, these are also immediate: the fact that the standard t-structure on $\Perf A$ has heart $\modu B$ (see Proposition \ref{prop:t}), and this is compatible with $\pi_*$, guarantees that
\begin{equation}
\sfH^0(\Gamma) = \Hom_{\sf{D}(A)}(\pi_*B, \pi_*B) \cong \Hom_{\sf{D}(B)}(B,B) \cong B 
\end{equation}
and
\begin{equation}
 \sfH^1(\Gamma) = \Hom_{\sf{D}(A)}(\pi_*B, \Sigma \pi_*B) \cong \Hom_{\sf{D}(B)}(B,\Sigma B) = 0.
\end{equation}
\end{proof}

Smoothness and properness are derived Morita invariant and so $\Gamma$ is again smooth and proper and, moreover, has the same noncommutative motive as $A$. Thus we obtain a description for the noncommutative motive of a number of simply connected coconnective dg algebras. The difficulty is in recognizing when our calculation applies to such a dg algebra; being augmented over $B=\sfH^0(\Gamma)$ is not automatic in general. 

\begin{ex}
Let $X$ denote a Burniat surface (we need to assume $k$ is algebraically closed of characteristic $\neq 2$). It was shown in \cite[Theorem 4.12]{alexeev2013derived} that there is a semi-orthogonal decomposition
\begin{equation}
\sf{D}^b(X)=\langle \mathcal{A},\mathcal{B}\rangle,
\end{equation}
where $\mathcal{A}=\langle \mathcal{L}_1, \ldots, \mathcal{L}_6\rangle$ is generated by an exceptional collection of line bundles $\mathcal{L}_i$ on $X$, and $\mathcal{B}$ is a quasi-phantom category: it has vanishing Hochschild homology and finite Grothendieck group. The category $\mathcal{A}$ is equivalent to the derived category of the dg-endomorphism algebra $A$ of $\bigoplus_{i=1}^6\mathcal{L}_i$ which is computed explicitly in \cite[Proposition 4.10]{alexeev2013derived}. It turns out $A$ is formal, and is of the form
\begin{equation}
\begin{tikzcd}[row sep=.2in,column sep=.6in]
\mathcal{L}_1 \ar["{3}" description]{rrd} \ar{r} \ar{rd} \ar{rdd} & \mathcal{L}_3 \ar["{2}" description]{rd} & \\
& \mathcal{L}_4 \ar["{2}" description]{r} & \mathcal{L}_6 \\
\mathcal{L}_2 \ar["{3}" description]{rru} \ar{ru} \ar{r} \ar{ruu} & \mathcal{L}_5 \ar["{2}" description]{ru} &
\end{tikzcd}
\end{equation}
with labels indicating the number of arrows. All arrows have degree $2$ and all non-trivial composites vanish, i.e.\ it is an algebra of radical square zero. So $A$ is smooth, proper, coconnective, and simply connected. Since $A^0 \cong k^6$ is semisimple, Corollary \ref{cor:coco} applies. Of course in this case knowledge of the noncommutative motive also follows immediately from the fact that $\mathcal{A}$ has a full exceptional collection.
\end{ex}



\section{Smoothness of $\sfH^0$}
\label{sec:h0}
Let us now make some comments on $\sfH^0(A)$ of a smooth, proper, and connective dg algebra; this is a point which bamboozled the authors for some time, and so we feel it is worth highlighting.

As some motivation, let us note that when $\sfH^0(A)$ is smooth one can avoid the technicalities of Section~\ref{sec:georeal} via the following special case of Theorem~\ref{thm:main}.

\begin{prop}\cite[Corollary 3.5]{Maaseik}
\label{thm:geometric}
For a smooth, proper, and connective dg algebra $A$ such that $\sfH^0(A)$ is smooth, there exists a fully faithful exact functor 
\begin{equation}
\label{eq:ffexact}
\Perf A \to \sf{D}^b(\coh(X)),
\end{equation} 
with left and right adjoints, for some smooth projective scheme $X$ with a full exceptional collection.
\end{prop}
\begin{proof}
By the discussion in \cite[Remark 3.6]{Maaseik}, smoothness of $\sfH^0(A)$ ensures one can apply Corollary 3.5 in loc.\ cit.\ to get a fully faithful exact functor \eqref{eq:ffexact} for some smooth projective scheme $X$. Though it is not mentioned there, $X$ can be chosen to admit a full exceptional collection. Indeed, the construction shows there is a fully faithful exact functor 
\begin{equation}
\Perf A \to \sf{T}=\langle \Perf \sf{H}^0(A),\ldots, \Perf \sf{H}^0(A)\rangle,
\end{equation} 
for some triangulated category $\sf{T}$ admitting a semi-orthogonal decomposition as indicated. By applying \cite[Remark 5.6 and Theorem 5.8]{MR3545926} there is a fully faithful exact functor
\begin{equation}
\sf{T} \to \sf{D}^b(\coh(X)),
\end{equation}
for some smooth projective scheme $X$ with a full exceptional collection. The composition of both functors gives \eqref{eq:ffexact}. The existence of left and right adjoints then follows from \cite[Propositions 3.17 and 3.24]{MR3545926}.
\end{proof}

Let us now examine smoothness of $\mathsf{H}^0(A)$ in more detail. First of all, for an arbitrary smooth and connective dg algebra $A$, it is easy to see that $\sfH^0(A)$ is not necessarily smooth.

\begin{ex}
\label{ex:dualnumbers}
Let $A=k\langle x,y \rangle$, with $|x|=0, |y|=-1$, and such that the differential is determined by $d(y)=x^2$. Then $A$ is smooth (since it is semi-free) and connective, but $\sfH^0(A)=k[x]/(x^2)$ is not smooth. 
\end{ex}

In the previous example, the dg algebra $A$ is not proper. Constructing smooth, proper and connective dg algebras with non-smooth $\mathsf{H}^0$ seems to be more subtle. However, contrary to what one might hope, examples do occur.

\begin{ex}
In \cite[\S 6.3]{MR3748359}, a smooth 15-dimensional algebra $\Lambda$ is presented, along with a 2-term silting complex $\mathbf{P}$. Hence, we get an equivalence 
\begin{equation}
\Perf \Lambda \simeq \Perf A,
\end{equation}
where $A=\RHom_{\Lambda}(\mathbf{P},\mathbf{P})$ is the dg endomorphism algebra of $\mathbf{P}$. In particular $A$ is smooth, proper and connective. In loc.\ cit.\ it is furthermore shown that $\sfH^0(A)=\End_{\sf{D}^b(\Lambda)}(\mathbf{P})$ is not smooth.
\end{ex}

We now show that not every finite dimensional algebra can occur as $\sf{H}^0$ of a smooth proper connective dg algebra. In the simplest case, if $\sfH^0(A) \cong k$, then not only is the noncommutative motive of $A$ isomorphic to $U(k)$, but the dg algebra is in fact $k$ in disguise. 

\begin{prop}\label{prop:notlocal}
If $A$ is smooth and proper, $\sfH^0(A) \cong k$ and $A$ is either connective, or coconnective, then $A$ is quasi-isomorphic to $k$.
\end{prop}
\begin{proof}
In the connective case we already know we can assume that $A^i=0$ for $i>0$. In the coconnective case we can assume that $A^0=k$ and $A^i=0$ for $i<0$, by \cite[Lemma 9.5]{efimov2009deformation}. Note that in both cases there is a map of dg algebras $A \to \sfH^0(A)\cong k$, so $\sfH^0(A)$ is a dg module over $A$.
 
Consider first the connective case. We have already noted in Corollary \ref{cor:cg} that $\sfH^0(A) \cong k \in \Perf A$. The main theorem of \cite{jorgensen10} now asserts that for any perfect complex $X \in \Perf A$ with non-zero cohomology, the amplitude 
\begin{equation}
\sf{amp}(X)=\sup \{i \mid \sfH^{-i}(X) \neq 0\} - \inf \{i \mid \sfH^{-i}(X) \neq 0\}
\end{equation}
of $X$ is at least the amplitude of $A$. Since $k \in \Perf A$, we hence deduce that $A$ has amplitude $0$, i.e. $\sfH^i(A) = 0$ if $i\neq 0$. But then $A$ must be formal and so there is a quasi-isomorphism $A\simeq \sfH^*(A) \cong k$.

The analogue of Corollary \ref{cor:cg} holds in the coconnective case and so, using \cite[Corollary 6.2]{MR2504967} as a replacement for J\o{}rgensen's result, the same amplitude argument works in the coconnective setting.
\end{proof}

\begin{cor}\label{cor:notlocal}
Let $k$ be an algebraically closed field. If $A$ is smooth, proper, and connective with $\sfH^0(A)$ local then $A$ is derived Morita equivalent to $k$.
\end{cor}
\begin{proof}
By Proposition \ref{prop:t}, $\Perf A$ has a bounded t-structure with heart $\modu \sfH^0(A)$. As $\mathsf{H}^0(A)$ is local it has a unique simple module $S$. Note that $S$ is a generator for $\Perf A$, indeed every object in the heart has finite length and every perfect complex can be glued together from its finitely many cohomology groups.

Thus $A$ is derived Morita equivalent to the dg-endomorphism algebra $\Lambda$ of $S$ as an $A$-module. In particular $\Lambda$ is smooth and proper. We deduce from the fact that $S$ lies in the heart of the standard t-structure that $\Lambda$ is coconnective. Since $S$ is simple and $k$ is algebraically closed we see that $\sfH^0(\Lambda)\cong k$. We are thus in the situation of Proposition~\ref{prop:notlocal} and conclude that $\Lambda$ is quasi-isomorphic to $k$.
\end{proof}

In contrast to Example~\ref{ex:dualnumbers} this shows that there is no smooth proper connective dg algebra with zeroth cohomology $k[x]/(x^2)$. It would be interesting to understand precisely which finite dimensional algebras can occur.



\begin{ack}
We're very grateful to Ben Antieau for a number of helpful comments on a preliminary version of this article, and are particularly indebted to Bernhard Keller for pointing out an erroneous claim made in a previous version (cfr. \S\ref{sec:h0}). We'd also like to express our gratitude to Dmitri Orlov for pointing out to us that Corollary~\ref{cor:fd} could be strengthened to a quasi-isomorphism. Finally, we would like to thank Michel Van den Bergh for comments on \S\ref{sec:georeals}.
\end{ack}




\bibliography{motive-bib}

\end{document}